\documentclass{amsart}

\usepackage{graphicx}

\newtheorem{theorem}{Theorem}[section]
\newtheorem{lemma}[theorem]{Lemma}

\theoremstyle{definition}
\newtheorem{definition}[theorem]{Definition}
\newtheorem{example}[theorem]{Example}

\newtheorem{conj}[theorem]{Conjecture}

\newtheorem{prop}[theorem]{Proposition}
\newtheorem{cor}[theorem]{Corollary}

\theoremstyle{remark}
\newtheorem{remark}[theorem]{Remark}

\numberwithin{equation}{section}

\begin{document}

\title{Singular hyperbolic metrics and negative subharmonic functions}

\author{Yu Feng}
\address{Wu Wen-Tsun Key Laboratory of Math, USTC, Chinese Academy of Sciences. \newline \indent School of Mathematical Sciences, University of Science and Technology of China,\newline \indent Hefei, 230026, China
}
\curraddr{Department of Mathematics, UC Berkeley, Berkeley,  CA 94720 USA}
\email{yuf@mail.ustc.edu.cn}
\thanks{
Y.F. is supported in part by China Scholarship Council.
Y.S. is supported in part by the National Natural Science Foundation of China
(Grant no. 11931009) and Anhui Initiative in Quantum Information Technologies (Grant no. AHY150200).
J.S. is partially supported by National Natural
Science Foundation of China (Grant No. 11721101).
B.X. is supported in part by the National Natural Science Foundation of China (Grant nos. 11571330 and 11971450).  Both Y.S. and B.X. are supported in part by
the Fundamental Research Funds for the Central Universities.
}
\thanks{$^\dagger$B.X. is the corresponding author.}

\author{Yiqian Shi}
\address{Wu Wen-Tsun Key Laboratory of Math, USTC, Chinese Academy of Sciences. \newline \indent School of Mathematical Sciences, University of Science and Technology of China,\newline \indent Hefei, 230026, China}
\email{yqshi@ustc.edu.cn}
\thanks{}

\author{Jijian Song}
\address{Center for Applied Mathematics. \newline \indent School of Mathematics, Tianjin University,\newline \indent Tianjin, 300350, China}
\email{smath@mail.ustc.edu.cn}
\thanks{}

\author{Bin Xu$^\dagger$}
\address{Wu Wen-Tsun Key Laboratory of Math, USTC, Chinese Academy of Sciences. \newline \indent School of Mathematical Sciences, University of Science and Technology of China,\newline \indent Hefei, 230026, China}
\email{bxu@ustc.edu.cn}
\subjclass[2010]{Primary 51M10; Secondary 34M35}

\date{}

\dedicatory{}

\keywords{singular hyperbolic metric, developing map, monodromy group, Zariski dense, subharmonic function}

\begin{abstract}
We propose a conjecture that the monodromy group of a singular hyperbolic metric on a non-hyperbolic Riemann surface is {\it Zariski dense} in ${\rm PSL}(2,\,{\Bbb R})$.
By using meromorphic differentials and affine connections,
we obtain an evidence of the conjecture that the monodromy group
of the singular hyperbolic metric can not be contained in four classes of one-dimensional Lie subgroups of ${\rm PSL}(2,\,{\Bbb R})$. Moreover, we confirm the conjecture if the Riemann surface is either one of the once punctured Riemann sphere, the twice punctured Riemann sphere, a once punctured torus and a compact Riemann surface.

\end{abstract}

\maketitle


\section{Introduction}\label{intro}

We investigate the monodromy groups of singular hyperbolic metrics on Riemann surfaces, not necessarily compact, in
terms of some analytic property of the underlying surfaces.
This project lies in the intersection of algebra, analysis and geometry on Riemann surfaces, and establishes a new connection between Differential Geometry and Potential Theory on Riemann surfaces.  We also use some techniques from some research works on cone spherical metrics during the course of the investigation.


\subsection{Background of singular hyperbolic metrics}
There have been many studies on singular  hyperbolic metrics. Nitsche \cite{N57}, Heins \cite{HE62}, Yamada \cite{Ya88}, Chou and Wan \cite{CW94,CW95} proved that an isolated singularity of a  hyperbolic metric must be either a cone singularity or a cusp one. Actually, if the curvature of a conformal Riemannian metric on a punctured neighbourhood is bounded below and above by two negative constants respectively, then the isolated singularity of the metric must be either a cone singularity or a cusp one (see \cite{HE62, KR08, MO93}). We gave the explicit expressions of  hyperbolic metrics near isolated singularities in \cite{FSX19, FSX19II} by using Complex Analysis.

Let $\Sigma$ be a Riemann surface, not necessarily compact, and $\textup{D}=\sum_{i=1}^\infty(\theta_{i}-1)p_{i}$ an ${\Bbb R}$-divisor on $\Sigma$ such that $0\leq \theta_i\not=1$, where $\{p_{i}\}_{i=1}^\infty$ is a discrete closed subset of $\Sigma$. We call $\mathrm{d} s^2$ a {\it singular hyperbolic metric
representing  {\rm D}} if and only if
\begin{itemize}

              \item   $\mathrm{d} s^2$ is a conformal metric of Gaussian curvature $-1$ on $\Sigma\backslash {\rm supp\, D}$, where we denote $\{p_i\}_{i=1}^\infty$ by ${\rm supp\, D}$.

              \item \ If $\theta_i>0$, then $\mathrm{d} s^2$\ {\it has a cone singularity at $p_{i}$ with cone angle\ $2\pi\theta_{i}>0$}. That is, in a neighborhood $U$ of  $p_{i}$, $\mathrm{d} s^2=e^{2u}\vert \mathrm{d} z \vert^2$, where $z$ is a complex coordinate of $U$  with\ $z(p_{i})=0$ and\ $u-(\theta_i-1)\ln \vert z \vert$\ extends continuously to $z=0$.

             \item \ If $\theta_i=0$, then $\mathrm{d} s^2$\ {\it has a cusp singularity at\ $p_{i}$}. That is, in a neighborhood $V$ of\ $p_{i}$,\ $\mathrm{d} s^2=e^{2u}\vert \mathrm{d} z \vert^2$, where\ $z$\ is a complex coordinate of $V$ with\ $z(p_{i})=0$ and\ $u+\ln\, \vert z \vert+\ln\,\left(-\ln\, \vert z \vert\right)$\ extends continuously to $z=0$.

 \end{itemize}

It was a classical problem about
the existence and uniqueness of a hyperbolic metric with finitely many prescribed singularities on a compact Riemann surface. By the Gauss-Bonnet formula, if\ $\mathrm{d} s^2$\ is a  hyperbolic metric representing the divisor\ $\textup{D}=\sum_{i=1}^n(\theta_{i}-1)p_{i}$ with $\theta_{i}\geq0$ on a compact Riemann surface $\Sigma$, then there holds\ $\chi(\Sigma)+\sum_{i=1}^n(\theta_{i}-1)<0$, where $\chi(\Sigma)$ is the Euler number of $\Sigma$. More than half a century ago, by using Potential Theory, M. Heins \cite{HE62} proved the following.

\begin{theorem}
\label{thm:rat}
There exists a unique  hyperbolic metric representing an $\mathbb{R}$-divisor\ $\textup{D}=\sum_{i=1}^n(\theta_{i}-1)p_{i}$ with $\theta_{i}\geq 0$ on a compact Riemann surface $X$ if and only if\ $\chi(X)+\sum_{i=1}^n(\theta_{i}-1)<0$.
\end{theorem}
\noindent Nearly three decades later, by using the PDE method,  both McOwen \cite{MO88} and Troyanov \cite{TR91}, who apparently did not know Heins' work, proved the same theorem for hyperbolic metrics with only cone singularities.

\subsection{Developing maps and subharmonic functions}

From the viewpoint of a combination of Complex Analysis and $G$-structure (\cite[Theorem 2.12]{GLX20} and \cite[Theorem 2.2]{LLX20}),  the concept of {\it developing map} is naturally associated to a singular hyperbolic metric. In this manuscript, we focus on the algebraic property of developing maps of singular hyperbolic metrics on Riemann surfaces, which turns out to be interwoven  with the analytic property of underlying surfaces. To give all the details of the story, we need to at first present the existence and the basic properties of developing maps for singular hyperbolic metrics.

\begin{theorem}{\rm{\cite[Theorem 2.2]{LLX20}}}
Let $\mathrm{d} s^2$ be a singular hyperbolic metric representing {\rm D} on a Riemann surface $\Sigma$. Then there exists a  multivalued locally univalent holomorphic map
$f:\Sigma\backslash {\rm supp\, D}\longrightarrow \mathbb{H}=\{z\in {\Bbb C}:\Im\,z>0\}$
called a {\rm developing map} of $\mathrm{d} s^2$ such that $\mathrm{d} s^2=f^{*}\mathrm{d} s_{0}^2$, where\ $\mathrm{d} s_{0}^2=\vert {\rm d}z \vert^{2}/(\Im z)^{2}$\ is the hyperbolic metric on the upper half-plane ${\Bbb H}$ and the monodromy representation of $f$ gives a homomorphism
${\frak M}_f:\pi_1\big(X\backslash {\rm supp\, D}\big)\to
\textup{PSL}(2,\mathbb{R})$,
whose image has well defined conjugacy class in $\textup{PSL}(2,\mathbb{R})$.
{\rm By abuse of notation, we just call the image of ${\frak M}_f$ the {\it monodromy group} of $\mathrm{d} s^2$.}
Any two developing maps of $\mathrm{d} s^2$\ are related by a fractional linear transformation in $\textup{PSL}(2,\mathbb{R})$.

\end{theorem}


We recall the classification of Riemann surfaces in terms of the existence
of a non-constant negative subharmonic function, which is a basic concept in Potential Theory.

\begin{definition}\rm{\cite[p.179]{FK92}}.
\label{defi:classification}
Let $\Sigma$ be a Riemann surface. We call $\Sigma$ \textit{elliptic} if and only if $\Sigma$ is compact. We call $\Sigma$ \textit{parabolic} if and only if $\Sigma$ is not compact and $\Sigma$ does not carry a negative non-constant subharmonic function. We call $\Sigma$ \textit{hyperbolic} if and only if $\Sigma$ carries a negative non-constant subharmonic function. We call $\Sigma$ {\it non-hyperbolic} in short if it is either elliptic or parabolic.
\end{definition}


\subsection{$G$-metrics}
We recall the following  well known fact about the classification of positive dimensional Lie subgroups
of ${\rm PSL}(2,\,{\Bbb R})$, by which we introduce the concept of {\it $G$-metric}.

\begin{prop}
\label{prop:classify}
{\it Up to conjugacy, we can classify all the positive dimensional proper Lie subgroups of ${\rm PSL}(2,\,{\Bbb R})$ as the following five classes{\rm :}
\begin{itemize}
\item the two-dimensional subgroup $
L=\left\{\begin{pmatrix} a & b \\ 0 & \frac{1}{a}\end{pmatrix}:a>0, b\in\mathbb{R}\right\}
$;

\item
the family $\{H_{1c}:c>0\}$ of one-dimensional Lie subgroups
$H_{1c}=\\\left\{\begin{pmatrix} c^{n} & t \\ 0 & c^{-n} \end{pmatrix}:n\in\mathbb{Z},\ t\in \mathbb{R}\right\}$;

\item the one-dimensional subgroup
$H_{2}=\left\{\begin{pmatrix} a & 0 \\ 0 & a^{-1}\end{pmatrix}:a>0 \right\}$;

\item the one-dimensional subgroup
$H_{2}'=\left\{\begin{pmatrix} a & 0 \\ 0 & a^{-1} \end{pmatrix}, \begin{pmatrix} 0 & b \\ -b^{-1} & 0 \end{pmatrix}:a>0, b>0 \right\}$;

\item the one-dimensional subgroup
$H_{3}=\left\{\begin{pmatrix} \cos t & \sin t \\ -\sin t & \cos t \end{pmatrix}:t\in \mathbb{R}\right\}/\{\pm I_{2}\}$.
\end{itemize}}
{\rm We denote by $L_0$ the normal subgroup
$
H_{11}=\left\{\begin{pmatrix} 1 & t \\ 0 & 1\end{pmatrix}:t\in\mathbb{R} \right\}
$
of $L$.}
\end{prop}

\begin{definition}
Let $\mathrm{d} s^2$ be a singular hyperbolic metric representing the divisor $\textup{D}=\sum_{i=1}^\infty(\theta_{i}-1)p_{i},\  0\leq \theta_{i}\not=1$ on a Riemann surface $\Sigma$, and $G$ a positive dimensional proper Lie subgroup of ${\rm PSL}(2,\, {\Bbb R})$.
We call $\mathrm{d} s^2$ a {\it $G$-metric} if
the monodromy group of $\mathrm{d} s^2$ lies in $G$.

\end{definition}


\subsection{Main results and the conjecture}
We state the main results and the conjecture of this manuscript as follows.
\begin{theorem}
\label{thm:ZC}
Let ${\rm d}s^2$ be a singular hyperbolic metric on a non-hyperbolic Riemann surface $\Sigma$. Then the following two statements hold:
\begin{itemize}
\item ${\rm d}s^2$ is not a $G$-metric if $G$ is a Lie subgroup of ${\rm PSL}(2,\,{\Bbb R})$ which is conjugate to $H_{2}$, $H_{2}'$, $H_{3}$ or $L_{0}$.

\item If $\Sigma$ is a compact Riemann surface, $\mathbb{C}$, $\mathbb{C}\backslash \{0\}$ or a punctured torus, then the monodromy group of ${\rm d}s^2$ is Zariski dense in $\textup{PSL}(2,\mathbb{R})$, i.e. it cannot be contained in any positive dimensional proper Lie subgroup of $\textup{PSL}(2,\mathbb{R})$.

\end{itemize}
\end{theorem}

This theorem stimulates us to propose the following.

\begin{conj}
\label{conj:ZP}
The monodromy group of a singular hyperbolic metric on a non-hyperbolic Riemann surface is  Zariski dense in $\textup{PSL}(2,\mathbb{R})$.
\end{conj}

\begin{remark}
The compact Riemann surface case of Theorem \ref{thm:ZC} can be thought of as an analogue of \cite[Theorem 7]{Falt83}, where G. Faltings proved the Zariski dense property in ${\rm PSL}(2,\,{\Bbb C})$ for the monodromy groups of permissible connections belonging to certain uniformization data on a compact Riemann surface.
\end{remark}

We investigate $G$-metrics as $G$ varies among all the positive dimensional proper Lie subgroups of $\textup{PSL}(2,\mathbb{R})$ in the remaining sections of this manuscript. In Section 2, we prove that there exists no $H_{3}$-metric on a non-hyperbolic Riemann surface, and construct a family of $H_{3}$-metrics with countably many cone singularities on the unit disc. We show the non-existence of either $H_2$-metric or $H_2'$-metric on a non-hyperbolic Riemann surface in Sections 3 and 4, respectively. In Section 5, we prove that there exists no $L$-metric on a compact Riemann surface, $\mathbb{C}$, $\mathbb{C}^*$ or a punctured torus. We show that any $L$-metric on the unit disk is automatically an $L_0$-metric in Section 5.
We prove Theorem \ref{thm:ZC} as a consequence of the results proved in Sections 2-5 and make a discussion for Conjecture \ref{conj:ZP} in the last section.

\section{$H_{3}$-metrics}
In this section,
we use the Poincar{\' e} disk model $\left(\mathbb{D}=\{z\in \mathbb{C}:\vert z \vert<1\}, \frac{4\vert {\rm d}z \vert^{2}}{(1-\vert z \vert^{2})^{2}}\right)$ other than the upper half plane model
$\left({\Bbb H},\, \vert {\rm d}z \vert^{2}/(\Im\, z)^{2} \right)$ to investigate an
$H_3$-metric ${\rm d}s^2$ representing an ${\Bbb R}$-divisor D on a Riemann surface $\Sigma$.  Hence, there exists a developing map $f:\Sigma\backslash {\rm supp\, D}\longrightarrow \mathbb{D}$ of the metric $\mathrm{d} s^2$ such that $\mathrm{d} s^2=f^{*}\left(4\vert {\rm d}z \vert^{2}/(1-\vert z \vert^{2})^{2}\right)$. Moreover, the monodromy of $f$ lies in
$$
{\rm U}(1)=\left\{z\mapsto e^{\sqrt{-1}t}:\,t\in [0,\,2\pi)\right\}.
$$
Hence we may also call the metric ${\rm d}s^2$ a ${\rm U}(1)$-{\it metric}.
Motivated by \cite{CWWX15}, we characterize a ${\rm U}(1)$-metric in terms of a meromorphic one-form on $\Sigma$ satisfying some geometric properties (Lemma \ref{lem:pq}). We can also construct a nonconstant bounded subharmonic function by using a developing map of a ${\rm U}(1)$-metric on $\Sigma$, which implies that $\Sigma$ must be hyperbolic
(Theorem \ref{thm:noH_{3}}). In addition, using some meromorphic one-forms, we can construct a family of ${\rm U}(1)$-metrics on ${\Bbb D}$ (Proposition \ref{prop:U(1)}).
\begin{lemma}
\label{lem:pq}
 Let\ $\mathrm{d} s^2$\ be a ${\rm U}(1)$-metric  representing an ${\Bbb R}$-divisor $\textup{D}=\sum_{i=1}^\infty(\theta_{i}-1)p_{i}$ on a Riemann surface $\Sigma$, and $f:\Sigma\backslash {\rm supp\, D}\longrightarrow \mathbb{D}$ a developing map of\ $\mathrm{d} s^2$ with monodromy  in ${\rm U}(1)$. Then the logarithmic differential
 $
 \omega:={\rm d} (\log f)=\frac{{\rm d} f}{f}
 $
 of\ $f$ extends to a meromorphic one-form on $\Sigma$ which satisfies the following properties{\rm :}
 \begin{enumerate}

\item If $p\in \Sigma\backslash {\rm supp\, D}$ is a pole of $\omega$, then $p$ is a simple pole of $\omega$ with residue 1.

\item $\mathrm{d} s^2$ has no cusp singularity. Moreover, if\ $\mathrm{d} s^2$ has a cone singularity at $p\in {\rm supp\, D}$ with cone angle $0<2\pi\alpha\notin 2\pi\,{\Bbb Z}$, then $p$ is a simple pole of $\omega$ with residue $\alpha$.

\item If $p$ is a cone singularity of $\mathrm{d} s^2$ with the angle $2\pi m\in 2\pi\,\mathbb{Z}_{>1}$, then $p$ is either a zero of $\omega$ with order $m-1$ or a simple pole of $\omega$ with residue $m$.

\item The real part $\Re\,\omega$ of $\omega$ is exact outside the set of poles of $\omega${\rm :}
$
 2\Re\omega={\rm d}(\log |f|^{2})
$.
\end{enumerate}
\end{lemma}

We call $\omega$ a {\it character 1-form} of the ${\rm U}(1)$-metric\ $\mathrm{d} s^2$.

\begin{proof} Since the developing map $f:\Sigma\backslash {\rm supp\,D}\to {\Bbb D}$ is
a multi-valued holomorphic function with monodromy in ${\rm U(1)}$, its logarithmic differential $\omega=\frac{{\rm d} f}{f}$ is a (single-valued) meromorphic one-form on $\Sigma\backslash {\rm supp\,D}$. Then we prove the four properties of $\omega$ in what follows, from which $\omega$ extends to a meromorphic one-form on $\Sigma$.

\begin{enumerate}

\item Suppose that $p\in \Sigma\backslash {\rm supp\,D}$ is a pole of $\omega$. We choose a function element $\mathfrak{f}$ near $p$.
    Since $\mathfrak{f}$ is a univalent holomorphic function near $p$, there exists a complex coordinate $z$ near $p$ with $z(p)=0$ such that $\mathfrak{f}=az+b,\ a\neq0$. Then $\omega=\frac{\mathfrak{f}'(z)}{\mathfrak{f}(z)}{\rm d}z=\frac{a}{az+b}{\rm d}z$. Since $p$ is a pole of $\omega$, we have $b=0$ and then $\mathfrak{f}=az,\ \omega=\frac{{\rm d}z}{z}$. Hence, $p$ is a simple pole of $\omega$ with residue 1.

 \item Since the developing map $f$ of $\mathrm{d} s^2$ has monodromy in ${\rm U}(1)$, $\mathrm{d} s^2$ has only cone singularities  (\rm{\cite[\S3]{FSX19}}). Suppose that $p$ is a cone singularity of $\mathrm{d} s^2$ with angle $0<2\pi\alpha\notin 2\pi\,{\Bbb Z}$. By \rm{\cite[Lemma 2.4]{FSX19}}, we can choose a function element $\mathfrak{f}$ near $p$ and a complex coordinate $z$ near $p$ such that $\mathfrak{f}=\frac{az^{\alpha}+b}{cz^{\alpha}+d}$ with $ad-bc=1$.  Since $f$ has monodromy in ${\rm U}(1)$,  there exists $\theta\in\mathbb{R}$ such that
$ e^{2\pi\sqrt{-1}\theta}\mathfrak{f}= e^{2\pi\sqrt{-1}\theta} \frac{az^{\alpha}+b}{cz^{\alpha}+d} = \frac{a e^{2\pi\sqrt{-1}\alpha}z^{\alpha}+b}{ce^{2\pi\sqrt{-1}\alpha} z^{\alpha}+d}.$
This is equivalent to the system:
$\left\{\  \begin{aligned} ac e^{2\pi\sqrt{-1}\alpha}(1-e^{2\pi\sqrt{-1}\theta}) &=0 \\
  (ade^{2\pi\sqrt{-1}\alpha}+bc)-e^{2\pi\sqrt{-1}\theta}(bce^{2\pi\sqrt{-1}\alpha}+ad) &=0\,. \\
  bd(1-e^{2\pi\sqrt{-1}\theta}) &=0 \end{aligned} \right.$
Solving it, we find that either\ $c=b=0$\ or\ $a=d=0$. If\ $a=d=0$, then\ $\mathfrak{f}=\frac{b}{cz^{\alpha}}$, which contradicts that $f$ takes values in ${\Bbb D}$. Thus\ $c=b=0$, that is,\ $\mathfrak{f}(z)$\ equals\ $\mu z^{\alpha}(\mu\neq0)$,\ $\mathfrak{f}(0)=0$. Hence $\omega=\frac{\alpha}{z}{\rm d}z$ and $p$ is a simple pole of $\omega$ with residue $\alpha$.

\item Suppose that $\mathrm{d} s^2$ has a cone singularity at $p$ with angle $2\pi m\in 2\pi\,{\Bbb Z}_{>1}$. By \rm{\cite[Lemma 2.4]{FSX19}}, we can choose a complex coordinate $z$ near $p$ such that $\mathfrak{f}=\frac{az^{m}+b}{cz^{m}+d}$ with $ad-bc=1$, and
$
\omega=\frac{\mathfrak{f}'(z)}{\mathfrak{f}(z)}{\rm d}z=\frac{mz^{m-1}}{(az^{m}+b)(cz^{m}+d)}{\rm d}z$.
If $bd\neq0$, then $p$ is a zero of $\omega$ with order $m-1$, and $\lim_{z\rightarrow p}f(z)=\frac{b}{d}\in\mathbb{D}\setminus\{0\}$.
If $bd=0$, then we have $b=0$ and $d\not=0$ since ${\frak f}$ takes values in ${\Bbb D}$. Hence,   $\mathfrak{f}(z)=\frac{az^{m}}{cz^{m}+d}$ and $p$ is a simple pole of $\omega$ with residue $m$.

\item By a simple computation, we have $2\Re\,\omega={\rm d}\,(\log\,f)+{\rm d}\,\overline {(\log\,f)}={\rm d}\log\, |f|^2$. By (1-3), $|f|^2$ is a single-valued smooth function outside the poles of $\omega$, where $\Re\,\omega$ is also exact.
\end{enumerate}

\end{proof}

\begin{remark} An $H_3$-metric with non-trivial monodromy has
a unique character 1-form.

\end{remark}

\begin{theorem}
\label{thm:noH_{3}}
There exists no $H_{3}-metric$ on a non-hyperbolic Riemann surface.

\end{theorem}

\begin{proof}
Let $\mathrm{d} s^2$ be an $H_3$-metric representing an ${\Bbb R}$-divisor D on a Riemann surface $\Sigma$ and
 $f:\Sigma\backslash {\rm supp\, D}\longrightarrow \mathbb{D}$ its developing map with monodromy in ${\rm U(1)}$. By the proof of Lemma \ref{lem:pq}, the function $|f|$ on $\Sigma\backslash {\rm supp\, D}$ extends continuously to $\Sigma$. If $p\in \Sigma$ is a cone singularity of $\mathrm{d} s^2$ with angle $0<2\pi\alpha\notin 2\pi\,{\Bbb Z}$, then by (2) of Lemma \ref{lem:pq},  we can choose a neighborhood $U$ of $p$ with complex coordinate $z$ such that $z(p)=0$ and $f(z)=z^{\alpha}$. Hence, $|f(z)|=|z|^{\alpha}$ is a subharmonic function on $U$. If $p\in \Sigma$ is either a smooth point or a cone singularity with angle $2\pi\, m\in 2\pi\,{\Bbb Z}_{>1}$ of $\mathrm{d} s^2$, by (1,3) of Lemma \ref{lem:pq},   we can choose a holomorphic function element $\mathfrak{f}$ in a neighborhood $V$ of $p$ such that $|f(z)|=|\mathfrak{f}(z)|$ is subharmonic on $V$. Therefore, $|f(z)|$ is a bounded non-constant subharmonic function on $\Sigma$, which implies that $\Sigma$ is a hyperbolic Riemann surface.

\end{proof}

Using certain meromorphic one-forms with simple poles and positive residues on the unit disc ${\Bbb D}$, we construct a family of ${\rm U}(1)$-metrics on ${\Bbb D}$ in the following.

\begin{prop}
\label{prop:U(1)}

{\it Let $\sum_{j=1}^{\infty}a_{j}$ be a convergent series of positive real numbers, $\{z_{1}, z_{2},\cdots\}$ a closed discrete subset of the unit disc ${\Bbb D}$, and $h:{\Bbb D}\to {\Bbb C}$ a holomorphic function such that $\Re \int^{z}\, h\, {\rm d}z$ has an upper bound.
Then
$\omega:=\bigg( \sum_{j=1}^{\infty}\frac{a_{j}}{z-z_{j}}+h(z) \bigg){\rm d}z$
is a meromorphic 1-form on the unit disc ${\Bbb D}$. And there exists a positive number $T$ and a 1-parameter family of ${\rm U}(1)$-metrics on ${\Bbb D}$,
$$\left\{{\rm d}\sigma_{\lambda}^{2}=:f_{\lambda}^{*}\left( \frac{4\vert {\rm d}z \vert^{2}}{(1-\vert z \vert^{2})^{2}}\right),\,\lambda\in(0,T)\right\}\quad{\rm where}\quad f_{\lambda}(z)=\lambda\cdot \exp\bigg(\int^{z}\omega\bigg),$$
such that $\omega$ is the common character 1-form of these metrics ${\rm d}\sigma_{\lambda}^{2}$.}
\end{prop}

\begin{proof}
Since $\sum_{j=1}^{\infty}a_{j}<\infty$, the series $\sum_{j=1}^{\infty}\frac{a_{j}}{z-z_{j}}$ is uniformly convergent in any compact subset $K$ of ${\Bbb D} \backslash \{z_{1},\ z_{2},\ \ldots\}$. Thus $\sum_{j=1}^{\infty}\frac{a_{j}}{z-z_{j}}+h(z)$ is a meromorphic function on ${\Bbb D}$ with simple poles $z_{1},\ z_{2},\ \ldots$ whose residues are $a_{1},\ a_{2},\ \ldots$, respectively. By the uniform convergence of $\sum_{j=1}^{\infty}\frac{a_{j}}{z-z_{j}}$ on a path in ${\Bbb D} \backslash \{z_{1},\ z_{2},\ \ldots\}$,  we cando the term-by-term integration and obtain
\[\begin{aligned}
\exp\left(\int^{z}\omega\right)= \prod\limits_{j=1}^{\infty}(z-z_{j})^{a_{j}}\cdot\left(\exp \int^{z}\, h{\rm d}z\right).
\end{aligned}
\]
Since $\sum_{j=1}^{\infty}a_{j}<\infty$ and $\Re \int^{z}h$ has an upper bound, there exists $M>0$ such that
\[
\bigg|\exp\int^{z}\omega\bigg|= \prod\limits_{j=1}^{\infty}|z-z_{j}|^{a_{j}}\cdot e^{\Re\int^{z}h}<2^{\sum_{j=1}^{\infty}a_{j}}e^{\Re\int^{z}h}<M.
\]
Solving the equation $
\omega={\rm d}\, (\log\, f)$ on ${\Bbb D} \backslash \{z_{1},\ z_{2},\ \ldots\}$, up to a complex multiple with modulus one, we obtain a one-parameter family of multi-valued locally univalent holomorphic functions
$
f_{\lambda}(z)=\lambda\cdot\exp\left(\int^{z}\omega\right),\,\lambda\in(0,\, T):=(0,\,1/M),
$
which take values in ${\Bbb D}$ and have
 monodromy in ${\rm U}(1)$. Hence,
 ${\rm d}\sigma_{\lambda}^{2}=:f_{\lambda}^{*}\left( \frac{4\vert {\rm d}z \vert^{2}}{(1-\vert z \vert^{2})^{2}}\right)$, $\lambda\in (0,\,T)$, form
 a one-parameter family of ${\rm U}(1)$-metrics with character 1-form $\omega$.

\end{proof}

\begin{remark}
The condition $\sum_{j=1}^{\infty}\,a_{j}<\infty$ is optimal in Proposition \ref{prop:U(1)}. In fact, if $\sum_{j=1}^{\infty}\, a_{j}=\infty$, then, by taking $z_j=1-1/(j+1)$,  we find that the series $\sum_{j=1}^{\infty}\frac{a_{j}}{z-z_{j}}$ diverges at $z=0$.
\end{remark}

\section{$H_{2}$-metrics}
Let\ $\mathrm{d} s^2$\ be an $H_2$-metric on a Riemann surface\ $\Sigma$. In this section, we can obtain a holomorphic 1-form on $\Sigma$ from $\mathrm{d} s^2$  (Lemma \ref{lem:H2}), and prove that $\Sigma$ must be hyperbolic (Theorem \ref{thm:H_{2}no}).  To this end, we need to recall some preliminary results.

\begin{prop}\rm{(\cite[Proposition 1.36]{AH2002})}
\label{lem:AH02}

{\it Suppose $X$ is path-connected, locally path-connected, and semilocally simply-connected. Then for every subgroup $H\subset\pi_{1}(X,x_{0})$ there is a covering space $p:X_{H}\rightarrow X$ such that $p_{*}(\pi_{1}(X_{H},\widetilde{x_{0}}))=H$ for a suitably chosen basepoint $\widetilde{x_{0}}\in X_{H}$.}

\end{prop}

\begin{definition}
For a path-connected, locally path-connected, and semilocally simply-connected space $X$, call a path-connected covering space $\widetilde{X}\rightarrow X$ {\it abelian} if it is normal and has abelian deck transformation group. In particular,
the commutator subgroup $[\pi_{1}(X),\pi_{1}(X)]$ determines a path-connected covering space $X^{\rm Ab}\stackrel{p}{\longrightarrow}X$ by Proposition \ref{lem:AH02}. Since the commutator subgroup is normal, $X^{\rm Ab}$ is a normal covering space. And the deck transformation  group of the covering $X^{\rm Ab}\stackrel{p}{\longrightarrow}X$ is isomorphic to $\pi_{1}(X)/[\pi_{1}(X),\pi_{1}(X)]$, which is abelian. Hence $X^{\rm Ab}$ is an abelian covering space, which is called the {\it maximal abelian covering} of $X$.

\end{definition}

\begin{lemma}\rm{(\cite[Theorem 2]{LS84})}
\label{lem:Sullivan}
{\it There exists no non-constant bounded harmonic function on any abelian cover of a non-hyperbolic Riemann surface.}
\end{lemma}

We introduce a holomorphic one-form from an $H_2$-metric  on a Riemann surface in the following.

\begin{lemma}
\label{lem:H2}
 Let\ $\mathrm{d} s^2$\ be an $H_2$-metric representing the divisor $\textup{D}=\sum_{i=1}^\infty(\theta_{i}-1)p_{i},\  0\leq \theta_{i}\not=1$ on a Riemann surface\ $\Sigma$ and $f:\Sigma\backslash {\rm supp\, D}\to \mathbb{H}$ a developing map of $\mathrm{d} s^2$ with monodromy in $H_{2}$. Then the following statements hold.
 \begin{enumerate}

\item The logarithmic differential
$
 \omega:=\frac{{\rm d}f}{f}
$
 of\ $f$ is a holomorphic 1-form on $\Sigma\backslash {\rm supp\, D}$, which extends to a holomorphic 1-form on $\Sigma$, which we call the {\rm character one-form} of $\mathrm{d} s^2$.

\item The singularities of\ $\mathrm{d} s^2$\  must be cone singularities with angles lying in $2\pi\,{\Bbb Z}_{>1}$. In particular, a cone singularity with angle
    $2\pi\,m\in 2\pi\,{\Bbb Z}_{>1}$ of
$\mathrm{d} s^2$ is a zero of $\omega$ with order $m-1$.

\item The developing map $f:\Sigma\backslash {\rm supp\, D}\to \mathbb{H}$ extends to a multi-valued holomorphic function on $\Sigma$ which also takes values in ${\Bbb H}$.

\end{enumerate}
\end{lemma}

\begin{proof}
\begin{enumerate}

\item Take a point $p\in \Sigma\backslash {\rm supp\, D}$ and choose a function element $\mathfrak{f}$ of $f$ near $p$. Since $f$ has monodromy in $H_{2}$ and takes values in
    ${\Bbb H}$, $\omega:=\frac{{\rm d}\mathfrak{f}}{\mathfrak{f}}$ does not depend on the choice of  $\mathfrak{f}$ and $\omega=\frac{{\rm d}f}{f}$ is a holomorphic one-form on
     $\Sigma\backslash {\rm supp\, D}$. We postpone to (2) the proof that $\omega$ extends to a holomorphic one-form on $\Sigma$.

\item Since $f$ has monodromy in $H_{2}$, $\mathrm{d} s^2$ has only cone singularities with angles in $2\pi\,{\Bbb Z}_{>1}$ (\rm{\cite[\S3]{FSX19}}). Suppose that it has a cone singularity at $p$ with angle $2\pi m\in 2\pi\,\mathbb{Z}_{>1}$. Then, by \rm{\cite[Lemma 2.4]{FSX19}}, we can choose a complex coordinate $z$ near $p$ such that $\mathfrak{f}=\frac{az^{m}+b}{cz^{m}+d}$ with $ad-bc=1$, and
$\omega=\frac{mz^{m-1}}{(az^{m}+b)(cz^{m}+d)}{\rm d}z$.  Since $f$ takes values in ${\Bbb H}$, we have $bd\neq 0$ which implies that $p$ is a zero of $\omega$ with order $m-1$, and $\lim_{z\rightarrow p}f(z)=\frac{b}{d}\in\mathbb{H}$. Hence, $\omega$ extends to a holomorphic one-form on $\Sigma$.

\item By (2), a function element of $f$ extends to a cone singularity of ${\rm d}s^2$ analytically and achieves a value in ${\Bbb H}$. Hence, $f$ extends to a multi-valued holomorphic function which also takes values in ${\Bbb H}$.
\end{enumerate}

\end{proof}

\begin{theorem}
\label{thm:H_{2}no}
There exists no $H_2$-metric on a non-hyperbolic Riemann surface.
\end{theorem}

\begin{proof}

Let $\mathrm{d} s^2$ be an $H_2$-metric on a Riemann surface $\Sigma$. Then its developing map $f:\Sigma\backslash {\rm supp\, D}\longrightarrow \mathbb{H}$ extends to a multi-valued holomorphic function on $\Sigma$ which also takes values in ${\Bbb H}$.
Consider the maximal abelian covering $\Sigma^{\rm Ab}\stackrel{p}{\longrightarrow}\Sigma$. Then $p_{*}(\pi_{1}(\Sigma^{\rm Ab},\widetilde{x_{0}}))=[\pi_{1}(\Sigma),\pi_{1}(\Sigma)]$ for a suitably chosen base point $\widetilde{x_{0}}\in\Sigma^{\rm Ab}$.
Let $[\gamma]$ be the homotopy class of a loop $\gamma$ based at $\widetilde{x_{0}}$ and $\mathfrak{f}$ a function element of $f$ near $x_0=p(\tilde{x_0})$. Then the analytical continuation $\mathfrak{f}\circ p_{[\gamma]}$ of ${\frak f}\circ p$ along $\gamma$ equals $\mathfrak{f}_{[p\circ \gamma]}=\mathfrak{f}_{[aba^{-1}b^{-1}]}=\mathfrak{f}$ for some two loops $a$ and $b$ based at $x_{0}=p(\tilde{x_0})$. Hence $f\circ p$ is a single-valued holomorphic function on $\Sigma^{\rm Ab}$ taking values in ${\Bbb H}$. Since both real and imaginary parts of $\frac{f\circ p-i}{f\circ p+i}$ are bounded harmonic functions on $\Sigma^{\rm Ab}$, $\Sigma$ must be a hyperbolic Riemann surface by Lemma \ref{lem:Sullivan}.

\end{proof}

\begin{remark} Let $\mathrm{d} s^2$ be an $H_{2}$-metric on the unit disc ${\Bbb D}$.
 Then it must represent an effective divisor D and its developing map $f$ extends to a multi-valued holomorphic function on ${\Bbb D}$. Since ${\Bbb D}$ is simply connected, $f$ is a single-valued holomorphic function on ${\Bbb D}$ and the effective divisor $(f)$ associated to $f$ coincides with D. Hence, the study of $H_{2}$-metrics on ${\Bbb D}$ is equivalent to that of the critical sets of analytic self-maps of ${\Bbb D}$. If ${\rm supp\, D}$ is a finite subset of ${\Bbb D}$, Heins {\rm{\cite[\S19]{HE62}}} observed that the developing map $f$ for ${\rm supp\, D}$ are precisely the finite Blaschke products with critical set ${\rm supp\, D}$. Kraus  {\rm{\cite{K12}}} gave a description of the critical sets of analytic self-maps of ${\Bbb D}$, which are countable in general.

\end{remark}

\section{$H_{2}'$-metrics}

In this section, we construct a meromorphic quadratic differential from an $H_2'$ metric on a Riemann surface $\Sigma$ (Lemma \ref{lem:H2'}) and prove that $\Sigma$ must be hyperbolic (Theorem \ref{thm:H_{2}'no}).

Firstly we recall some basic knowledge of quadratic differentials. A quadratic differential $q$ on $\Sigma$ is a section of $K_{\Sigma}\otimes K_{\Sigma}$, a differential of type (2,0) locally of the form $\phi(z){\rm d}z^{2}$, where $K_{\Sigma}$
is the canonical line bundle of $\Sigma$. It is said to be holomorphic (or meromorphic)
when $\phi(z)$ is holomorphic (or meromorphic).

\begin{lemma}
\label{lem:H2'}
 Let\ $\mathrm{d} s^2$\ be an $H_2'$-metric on a Riemann surface\ $\Sigma$, representing the divisor $\textup{D}=\sum_{i=1}^\infty(\theta_{i}-1)p_{i},\  0\leq \theta_{i}\not=1$. Let $f:\Sigma\backslash {\rm supp\, D}\to \mathbb{H}$\ be a developing map of\ $\mathrm{d} s^2$, whose monodromy lies in $H_{2}'$. Then
$
 q:=\left(\frac{{\rm d}f}{f}\right)^{\otimes 2}
$
 is a holomorphic quadratic differential on $\Sigma\backslash {\rm supp\,D}$ which extends to a meromorphic quadratic differential on $\Sigma$, called the {\rm character quadratic differential} of $\mathrm{d} s^2$. Moreover, we have the following.

 \begin{enumerate}
\item The singularities of\ $\mathrm{d} s^2$\  must be cone singularities. Suppose that\ $\mathrm{d} s^2$ has a cone singularity at $p\in {\rm supp\,D}$ with the angle $2\pi\alpha>0$, where $\alpha$ is a non-integer. Then $\alpha=\frac{1}{2}+k$, where $k\in\mathbb{Z}_{\geq 0}$. If $k=0$, then $p$ is a simple pole of $q$; if $k>0$, then $p$ is a zero of $q$ with order $2k-1$.

\item If\ $\mathrm{d} s^2$ has a cone singularity at $p\in {\rm supp\, D}$ with angle $2\pi \in\mathbb{Z}_{>1}$, then $p$ is a zero of $q$ with order $2m-2$.

\item The ${\Bbb Z}$-divisor $(q)$ associated to $q$ is relevant to D by the equation $(q)=2 \cdot {\rm D}$.
\end{enumerate}
\end{lemma}

\begin{proof}
Suppose that $p\in \Sigma\backslash {\rm supp\, D}$, we choose a function element $\mathfrak{f}$ of $f$ near $p$. Since $f$ has  monodromy in $H_{2}'$ and takes values in ${\Bbb H}$, $q:=\left(\frac{d\mathfrak{f}}{\mathfrak{f}}\right)^{\otimes 2}$ does not depend on the choice of $\mathfrak{f}$ and $q=\left(\frac{{\rm d}f}{f}\right)^{\otimes 2}$ is a holomorphic quadratic differential on $\Sigma\backslash {\rm supp\, D}$.
We postpone to (1-2) the proof that $q$ extends to a meromorphic quadratic differential on $\Sigma$.

\begin{enumerate}

\item We show at first that {\it $\mathrm{d} s^2$ has only cone singularities} in the following. Suppose that $p\in {\rm supp\,D}$ is a cusp one of $\mathrm{d} s^2$. By \rm{\cite[Lemma 2.4]{FSX19}}, we can choose a function element $\mathfrak{f}$ near $p$ and a complex coordinate $z$ near $p$ such that $\mathfrak{f}=\frac{a\log z+b}{c\log z+d}$ with $ad-bc=1$. Since $f$ has monodromy $H_{2}'$,  there exists $\lambda>0$ such that
$-\frac{\lambda}{\frac{a\log z+b}{c\log z+d}}= \frac{a(\log z+2\pi\sqrt{-1})+b}{c(\log z+2\pi\sqrt{-1})+d}$,
or there exists $\lambda'>0$ such that
$\lambda'{\frac{a\log z+b}{c\log z+d}}= \frac{a(\log z+2\pi\sqrt{-1})+b}{c(\log z+2\pi\sqrt{-1})+d}.$ However, neither of these two equations has any solutions.

Suppose that\ $\mathrm{d} s^2$ has a cone singularity at $p\in \rm {\rm supp\,D}$ with angle $2\pi\alpha>0$, where $\alpha$ is a non-integer. Then we can choose a function element $\mathfrak{f}$ near $p$ and a complex coordinate $z$ near $p$ such that $\mathfrak{f}=\frac{az^{\alpha}+b}{cz^{\alpha}+d}$ with $ad-bc=1$. Since $f$ has monodromy in $H_{2}'$, then either there exists $0<\lambda'\neq 1$ such that
$ \lambda'\mathfrak{f}= \lambda'\frac{az^{\alpha}+b}{cz^{\alpha}+d} = \frac{a e^{2\pi\sqrt{-1}\alpha}z^{\alpha}+b}{ce^{2\pi\sqrt{-1}\alpha} z^{\alpha}+d}$,
or there exists $\lambda>0$ such that
$-\frac{\lambda}{\mathfrak{f}}= -\frac{\lambda}{\frac{az^{\alpha}+b}{cz^{\alpha}+d}}= \frac{a e^{2\pi\sqrt{-1}\alpha}z^{\alpha}+b}{ce^{2\pi\sqrt{-1}\alpha}z^{\alpha}+d}$.
However, there is no solution for the former equation, and the latter one is equivalent to the system
$\left\{\  \begin{aligned} a^{2}+\lambda c^{2} &=0 \\
(ab+\lambda cd)\big(1+e^{2\pi\sqrt{-1}\alpha}\big) &=0\,. \\
 b^{2}+\lambda d^{2} &=0 \end{aligned} \right.$
If $1+e^{2\pi\sqrt{-1}\alpha}\neq 0$, then there is no solution for this system. We must have  $1+e^{2\pi\sqrt{-1}\alpha}=0$ and then $\alpha=\frac{1}{2}+k$ for some $k\in\mathbb{Z}_{\geq 0}$. Then we obtain from the system that $ad=-bc=1/2$,
$\mathfrak{f}= \frac{az^{\alpha}+b}{cz^{\alpha}+d},\   \mathfrak{f}'= \frac{\alpha z^{\alpha-1}}{(cz^{\alpha}+d)^{2}},$
$\mathfrak{f}(0)=\frac{b}{d}\in\mathbb{H},$ and
$
q=\frac{\alpha^{2}z^{2\alpha-2}}{(acz^{2\alpha}+bd)^{2}}{\rm d}z^{2}.
$
Therefore, if $k=0$, then $p$ is a simple pole of $\varphi$; if $k>0$, then $p$ is a zero of $\varphi$ with order $2k-1$.

\item Suppose that\ $\mathrm{d} s^2$ has a cone singularity at $p\in \rm supp\,D$ with  angle $2\pi m\in 2\pi\,\mathbb{Z}_{>1}$. Then
 we choose a complex coordinate $z$ near $p$ such that $\mathfrak{f}=\frac{az^{m}+b}{cz^{m}+d}$ with $ad-bc=1$ and
$q=\frac{m^{2}z^{2m-2}}{(az^{m}+b)^{2}(cz^{m}+d)^{2}}{\rm d}z^{2}$.
Since $f$ takes values in ${\Bbb H}$, we have $bd\neq 0$, which implies that $p$ is a zero of $\varphi$ with order $2m-2$, and $\lim_{z\rightarrow p}f(z)=\frac{b}{d}\in\mathbb{H}$.

\item The equation $(q)=2\cdot {\rm D}$ follows from (1-2).
\end{enumerate}

\end{proof}

\begin{theorem}\rm{(\cite[p.181]{FK92})}.
\label{thm:Green}
{\it Let $\Sigma$ be a Riemann surface and $p$ a point on $\Sigma$. There exists the Green function with singularity $p$ on $\Sigma$
if and only if $\Sigma$ is hyperbolic.}
\end{theorem}

\begin{theorem}
\label{thm:H_{2}'no}
There exists no $H_2'$-metric on a non-hyperbolic Riemann surface.
\end{theorem}
\begin{proof}
Let $\mathrm{d} s^2$ be an $H_2'$-metric on a Riemann surface $\Sigma$, $f:\Sigma\backslash {\rm supp\, D}\to \mathbb{H}$ a developing map of it with monodromy in $H_2'$ and
$q=\left(\frac{{\rm d}f}{f}\right)^{\otimes 2}$ the character quadratic differential of $\mathrm{d} s^2$.  By Lemma \ref{lem:H2'}, we have $(q)=2{\rm D}$, where $q$ has at worst simple poles.
As the proof of \rm{\cite[Lemma 3.1]{SCLX18}}, the quadratic differential $q$ induces the canonical double cover $\pi:\widehat{\Sigma}\rightarrow\Sigma$, branching over critical points of $q$ whose orders are odd, such that $\pi^{*}\varphi=\omega^{\otimes 2}$, where $\omega$ is a holomorphic one-form on $\widehat{\Sigma}$. Define $\widehat{f}:=f\circ \pi$, which is a multi-valued holomorphic function on $\widehat{\Sigma}$. Since $\pi^{*}q=\pi^{*}\left(\frac{{\rm d}f}{f}^{\otimes 2}\right)=\left(\frac{d\widehat{f}}{\widehat{f}}\right)^{\otimes 2}=
\omega^{\otimes 2}$, we have $\omega={\rm d}(\log\widehat{f})$ (up to sign). Hence,  $\widehat{f}$ has monodromy in $H_{2}$ and is a developing map of the pull-back metric $\pi^*\mathrm{d} s^2$ on $\widehat{\Sigma}$, which is then an $H_2$-metric.
By Theorem \ref{thm:H_{2}no}, $\widehat{\Sigma}$ is a hyperbolic Riemann surface.

By Theorem \ref{thm:Green}, we cantake the Green function $G$ on $\widehat{\Sigma}$ with singularity $p\in \widehat{\Sigma}$. Then $G$ is a positive harmonic function on $\Sigma\backslash \{p\}$, $G+\log\,|z|$ is harmonic near $p$, where $z$ is a complex coordinate centered at $p$. And $G$ is the minimal function satisfying the aforementioned two properties. Hence, $u:=\max\ \{-1, -G\}$ is a negative subharmonic function on $\widehat{\Sigma}$ and equals constant $-1$ near $p$. Moreover, it is neither a constant nor harmonic by the minimal property of $G$.
Note that the double cover $\pi$ is a branched Galois cover whose deck transformation group is generated by the holomorphic involution $\tau$ on $\widehat{\Sigma}$, where $\tau$ acts on $\pi^{-1}(x)$ as a swap for all point $x$ outside the critical points of $q$ whose orders are odd.   Then
$F(x):=\frac{u(y)+u\circ\tau(y)}{2}$,
where $x\in \Sigma,\ y\in \pi^{-1}(x)$, is a negative subharmonic function on $\Sigma$.
Since $u$ is not harmonic on $\Sigma$, $F$ is not a constant on $\widehat{\Sigma}$, which implies that $\Sigma$ is a hyperbolic Riemann surface.

\end{proof}

\begin{remark}
Though Theorem \ref{thm:H_{2}no} is contained in Theorem \ref{thm:H_{2}'no}, we intentionally spend the preceding section to narrate the former theorem and its preliminary lemma (Lemma \ref{lem:H2}) in detail since they have independent interest.
\end{remark}

\section{$L$-metrics}

In this section, we investigate an $L$-metric $\mathrm{d} s^2$ on a Riemann surface $\Sigma$.
In the first subsection, we find that $\mathrm{d} s^2$ induces an affine connection on $\Sigma$ (Lemma \ref{lem:L}), by which we prove that there exists no $L$-metric on a compact Riemann surface (Corollary \ref{Cor:2}). In the second one, we prove that there exists no $L$-metric on $\mathbb{C}$, $\mathbb{C}-\{0\}$ or a punctured torus.  We pay special attention to $L_0$-metrics in the last subsection and prove the non-existence of them on a non-hyperbolic Riemann surface.

\subsection{There exists no $L$-metric on a compact Riemann surface}

\begin{definition}(\rm{\cite{RC67}}, \rm{\cite[p.264]{RM72}})
Let $\Sigma$ be a Riemann surface, $\{U_{\alpha},z_{\alpha}\}$  a complex atlas on $\Sigma$, and
$\psi_{\alpha\beta}=z_{\alpha}\circ z_{\beta}^{-1}:z_{\beta}(U_{\alpha}\cap U_{\beta})\rightarrow z_{\alpha}(U_{\alpha}\cap U_{\beta})$
the coordinate transition functions. We call a family of meromorphic functions $\{h_{\alpha}:U_{\alpha}\rightarrow\overline{\mathbb{C}}=\mathbb{C}\cup\infty\}$ an
 {\it affine connection} on $\Sigma$ if for all $p\in U_{\alpha}\cap U_{\beta}$ there holds
$h_{\beta}\big(z_{\beta}(p)\big)=h_{\alpha}\big(z_{\alpha}(p)\big)\cdot\left(\frac{{\rm d}z_{\beta}}{{\rm d}z_{\alpha}}\right)^{-1}+\theta_{1}(\psi_{\alpha\beta}\big(z_{\beta}(p))\big)$,
where $\theta_{1}\psi(z)=\frac{\psi''(z)}{\psi'(z)}$.

\end{definition}

Let $h=\{h_{\alpha}\}$ be an affine connection on $\Sigma$. For each point $p$ in $U_{\alpha}$, we choose a small loop  $\Gamma$ winding around $p$ on counterclockwise and define the {\it residue} of $h$ at $p$ by
${\rm Res}(h,\,p)=\frac{1}{2\pi i}\int_{\Gamma}h_{\alpha}{\rm d}z_{\alpha}$,
and the {\it residue} of $h$ by
${\rm Res}(h)=\sum_{p\in \Sigma}{\rm Res}(h,p)$
provided that there exist only finitely many nonzero summands.
Note that ${\rm Res}(h,p)$ is independent of the choice of representatives for $h$
(\rm{\cite[p.270]{RM72}}).

\begin{lemma}\rm{\cite[Lemma 2]{RM72}}
\label{lem:affine}
{\it Let $h=\{h_{\alpha}\}$ be an affine connection on a compact Riemann surface $X$. Then ${\rm Res}(h)=-\chi(X)$, where $\chi(X)$ is the Euler number of $X$.}

\end{lemma}

\begin{lemma}
\label{lem:L}
Let\ $\mathrm{d} s^2$\ be an $L$-metric on a Riemann surface\ $\Sigma$, representing the divisor $\textup{D}=\sum_{i=1}^\infty(\theta_{i}-1)p_{i},\  0\leq \theta_{i}\not=1$. Let $f:\Sigma\backslash {\rm supp\, D}\longrightarrow \mathbb{H}$\ be a developing map of\ $\mathrm{d} s^2$, whose monodromy lies in $L$. Let $\{U_{\alpha},z_{\alpha}\}$ be a complex atlas on $\Sigma$ and
$\psi_{\alpha\beta}=z_{\alpha}\circ z_{\beta}^{-1}:z_{\beta}(U_{\alpha}\cap U_{\beta})\rightarrow z_{\alpha}(U_{\alpha}\cap U_{\beta})$ the coordinate transition functions. Then we have the following.
\begin{enumerate}
\item
$
h:=\{h_{\alpha}:=f''/f':U_{\alpha}\rightarrow\overline{\mathbb{C}}\}
$
is an affine connection on $\Sigma\backslash {\rm supp\, D}$, which extends to an affine connection on $\Sigma$.

\item A singularity of\ $\mathrm{d} s^2$\ is either a cusp singularity or a cone one with angle $2\pi m\in 2\pi\,\mathbb{Z}_{>1}$. Moreover, both a cusp singularity and a cone one with angle $2\pi m$ are simple poles of $f''/f'$, where the residues of $h$ are $-1$ and $m-1$, respectively.
\end{enumerate}

\end{lemma}

\begin{proof}
\begin{enumerate}\item
We choose a point $p$ in $\big(U_\alpha\cap U_\beta\big)\backslash {\rm supp\, D}$ and
take a function element $\mathfrak{f}$ of $f$ in a neighborhood of $p$ in $U_\alpha\cap U_\beta$, where $f''/f'$ is a single-valued function since $f$ has monodromy in $L$. Since
$
\mathfrak{f}'(z_{\beta}(p))=\mathfrak{f}'\big(z_{\alpha}(p)\big)
\psi_{\alpha\beta}'\big(z_{\beta}(p)\big)
$ and
$
\mathfrak{f}''\big(z_{\beta}(p)\big)=
\mathfrak{f}'\big(z_{\alpha}(p)\big)\psi_{\alpha\beta}''\big(z_{\beta}(p)\big)
+\mathfrak{f}''\big(z_{\alpha}(p)\big)\Big(\psi_{\alpha\beta}'\big(z_{\beta}(p)\big)\Big)^{2},
$
we have
\[
 \begin{aligned}
h_{\beta}\big(z_{\beta}(p)\big)&
=\frac{\mathfrak{f}''\big(z_{\beta}(p)\big)}
{\mathfrak{f}'\big(z_{\beta}(p)\big)}
=\frac{\psi_{\alpha\beta}''\big(z_{\beta}(p)\big)}{\psi_{\alpha\beta}'\big(z_{\beta}(p)\big)}
+\frac{\mathfrak{f}''\big(z_{\alpha}(p)\big)}{\mathfrak{f}'\big(z_{\alpha}(p)\big)}
\psi_{\alpha\beta}'\big(z_{\beta}(p)\big)\\
&=\theta_{1}\Big(\psi_{\alpha\beta}\big(z_{\beta}(p)\big)\Big)+
h_{\alpha}\big(z_{\alpha}(p)\big)\left(\frac{{\rm d}z_{\alpha}}{{\rm d}z_{\beta}}\right)(p).
\end{aligned}
\]
Hence, $f''/f'$ defines an affine connection on $\Sigma\backslash {\rm supp\, D}$. We postpone to (2) the proof that it extends to $\Sigma$.

\item
Suppose that\ $\mathrm{d} s^2$ has a cone singularity at $p$ with angle $0<2\pi\alpha\notin 2\pi\,{\Bbb Z}$. By \rm{\cite[Lemma 2.4]{FSX19}}, we can choose a function element $\mathfrak{f}$ near $p$ and a complex coordinate $z$ near $p$ such that $\mathfrak{f}=\frac{az^{\alpha}+b}{cz^{\alpha}+d}$ with $ad-bc=1$. Since $f$ has  monodromy in $L$, there exist $s>0, t\in\mathbb{R}$ such that
$ s\mathfrak{f}+t= s\frac{az^{\alpha}+b}{cz^{\alpha}+d}+t = \frac{a e^{2\pi\sqrt{-1}\alpha}z^{\alpha}+b}{ce^{2\pi\sqrt{-1}\alpha} z^{\alpha}+d}$, which is equivalent to the system:
$\left\{\  \begin{aligned} c(a-sa-tc)e^{2\pi\sqrt{-1}\alpha} &=0, \\
  ade^{2\pi\sqrt{-1}\alpha}+bc &=(sad+tcd)+e^{2\pi\sqrt{-1}\alpha}(sbc+tcd), \\
  (b-sb-td)d &=0. \end{aligned} \right.$
If $c=0$, then $ad=1$. Solving the system, we obtain $e^{2\pi\sqrt{-1}\alpha}=s>0$, $s=1$ and $\alpha\in {\Bbb Z}$. Contradiction!
If $c\neq0$, then, by solving the system, we have $a-sa=tc$ and $se^{2\pi\sqrt{-1}\alpha}=1$.
Since $s>0$, we obtain $s=1, t\neq 0$ and $c=d=0$. Contradiction!

Suppose that $p\in {\rm supp\, D}$ is a cusp singularity of $\mathrm{d} s^2$. By \rm{\cite[Lemma 2.4]{FSX19}}, we can choose a function element $\mathfrak{f}$ near $p$ and a complex coordinate $z$ near $p$ such that $\mathfrak{f}=\frac{a\log z+b}{c\log z+d}$ with $ad-bc=1$. Since the monodromy of $f$ belongs to $L$, then there exists $s>0, t\in\mathbb{R}$ such that
$ s{\mathfrak{f}}+t=s{\frac{a\log z+b}{c\log z+d}}+t= \frac{a(\log z+2\pi\sqrt{-1})+b}{c(\log z+2\pi\sqrt{-1})+d}$, which is equivalent to
the system:
$\left\{\  \begin{aligned} c(a-sa-tc) &=0, \\
  ad+ac2\pi\sqrt{-1}+bc &=c(sb+td)+(c2\pi\sqrt{-1}+d)(sa+tc), \\
  ad2\pi\sqrt{-1}+bd &=(c2\pi\sqrt{-1}+d)(sb+td). \end{aligned} \right.$
It is easy to check that the system has no solution if $c\not=0$. Hence we have $c=0$, $ad=1$ and $\mathfrak{f}=a^{2}\log z+ab$. Since $\Re\, \log z<0$ as $|z|<<1$ and ${\frak f}$ takes values in ${\Bbb H}$, we obtain $a^{2}=-\sqrt{-1}r$ for some $r>0$. Hence,
$h_{\alpha}=\frac{\mathfrak{f}''}{\mathfrak{f}'}=-\frac{1}{z}$ in a neighborhood $u_{\alpha}$ of $p$ and ${\rm Res}(h,\, p)=-1$.

Suppose that\ $\mathrm{d} s^2$ has a cone singularity at $p$ with angle $2\pi m\in 2\pi\,\mathbb{Z}_{>1}$. Then we can choose a function element $\mathfrak{f}$ near $p$ and a complex coordinate $z$ near $p$ such that $\mathfrak{f}=\frac{az^{m}+b}{cz^{m}+d}$ with $ad-bc=1$. Since ${\frak f}$ is an open map, we have $\lim_{z\rightarrow p}{\frak f}(z)=\frac{b}{d}\in\mathbb{H}$ and $bd\not=0$.  Hence, we obtain $\mathfrak{f}'(z)=\frac{mz^{m-1}}{(cz^{m}+d)^{2}}$, $\frac{\mathfrak{f}''}{\mathfrak{f}'}=\frac{m-1}{z}-\frac{2cmz^{m-1}}{cz^{m}+d}(d\neq0)$, and ${\rm Res}(h,p)=m-1$.

Therefore,
$
h=\{h_{\alpha}=\frac{f''}{f'}:U_{\alpha}\rightarrow\overline{\mathbb{C}}\}
$
extends to  an affine connection on $\Sigma$.

\end{enumerate}
\end{proof}

\begin{cor}
\label{Cor:2}
{\it There exists no $L$-metric on a compact Riemann surface $X$.}
\end{cor}
\begin{proof}
Suppose that $\mathrm{d} s^2$ is an $L$-metric representing an ${\Bbb R}$-divisor ${\rm D}$ on $X$. Then, by Lemma \ref{lem:L},
it has either cusp singularities or cone singularities with angles lying in $2\pi\,{\Bbb Z}_{>1}$.  Hence, the divisor D has form $\textup{D}=\sum_{i=1}^j(m_{i}-1)p_{i}+\sum\limits_{i=j+1}^n(-1)p_{i},\ \text{where}\, m_i\in  \mathbb{Z}_{>1}$. By Lemma \ref{lem:L}, we have ${\rm Res}(h,p_{i})=m_{i}-1$ as $1\leq i\leq j$, and ${\rm Res}(h,p_{i})=-1$ as $i>j$. By Lemma \ref{lem:affine}, $-\chi(X)=\sum_{i=1}^j(m_{i}-1)-(n-j)$, which contradicts  Theorem \ref{thm:rat}.

\end{proof}

\subsection{Non-existence of $L$-metrics on $\mathbb{C}$, $\mathbb{C}\backslash \{0\}$ or a punctured torus}
~\\
To prove this, we need the following.

\begin{theorem}\label{thm:UL}
{\rm (\cite[p.298]{UL01})}
Let $\Sigma$ be a Riemann surface such that none of its abelian covers is hyperbolic.
Then $\Sigma$ is the Riemann sphere, ${\Bbb C}$, ${\Bbb C}\backslash \{0\}$, a torus or a punctured torus.
\end{theorem}

\begin{theorem}
\label{thm:no L on C^{*}}
There exists no $L$-metric on $\mathbb{C}$, $\mathbb{C}\backslash \{0\}$ or a punctured torus.
\end{theorem}
\begin{proof}

Let $\mathrm{d} s^2$ be an $L$-metric on a Riemann surface $\Sigma$ and  $f:\Sigma\backslash {\rm supp\, D}\to \mathbb{H}$ a developing map of it with monodromy  in $L$. By Lemma \ref{lem:L}, the singularity of\ $\mathrm{d} s^2$ is either a cusp singularity or a cone one with angle $2\pi m\in\mathbb{Z}_{>1}$.
By the proof of Lemma \ref{lem:L}, if $p\in {\rm supp\, D}$ is a cusp singularity of $\mathrm{d} s^2$, then we can choose a function element $\mathfrak{f}$ near $p$ and a complex coordinate $z$ near $p$ such that $\mathfrak{f}=-\sqrt{-1}r\log z+s$, where $r>0$ and $s\in\mathbb{R}$. Since $-\Im\,\mathfrak{f}=r\log|z|$ extends to a negative subharmonic function in a neighborhood of $p$, $-\Im\, f$ forms a multi-valued negative subharmonic function on $\Sigma$ with monodromy in $H_{2}$. Taking the maximal abelian covering ${\Sigma}^{\rm Ab}\stackrel{\pi}{\longrightarrow}\Sigma$, we obtain  a negative non-constant subharmonic function $(-\Im\,f)\circ \pi$  on ${\Sigma}^{\rm Ab}$, which implies that ${\Sigma}^{\rm Ab}$ is a hyperbolic Riemann surface. By Theorem \ref{thm:UL}, the Riemann surface $\Sigma$ is not isomorphic to  $\overline{\mathbb{C}}$, $\mathbb{C}$, $\mathbb{C}\backslash \{0\}$, a torus, or a punctured torus.

\end{proof}

\subsection{$L_0$-metric}
~\\
In this subsection, we obtain a meromorphic one-form satisfying some geometric properties from an $L_0$ metric on a Riemann surface (Lemma \ref{lem:L_0}), by which we show the non-existence of $L_0$-metrics on a non-hyperbolic Riemann surface (Theorem \ref{thm:no L_{0}}). To this end, we need a preliminary lemma as follows.

\begin{lemma}{\rm (\cite[Theorem 3.6.1]{TR95})}
\label{thm:TR95} Let $U$ be an open subset of $\mathbb{C}$, $E$ a closed polar set, and $u$  a subharmonic function on $U\setminus E$. Suppose that each point of $U\setminus E$ has a neighbourhood $N$ such that $u$ is bounded above on $N\backslash E$. Then $u$ extends uniquely to a subharmonic function $U$.
\end{lemma}

\begin{lemma}
\label{lem:L_0}
 Let\ $\mathrm{d} s^2$\ be an $L_0$-metric on a Riemann surface\ $\Sigma$, representing the divisor $\textup{D}=\sum_{i=1}^\infty(\theta_{i}-1)p_{i}$, and $f:\Sigma\backslash {\rm supp\, D}\to \mathbb{H}$ a developing map of\ $\mathrm{d} s^2$ with monodromy in $L_{0}$. Then
 $\omega:={\rm d}f$
 is a holomorphic 1-form on $\Sigma\backslash {\rm supp\, D}$ and extends to a meromorphic one-form on $\Sigma$, which we call the {\rm character one-form} of $\mathrm{d} s^2$. Moreover, we have the following.

 \begin{enumerate}

\item If\ $\mathrm{d} s^2$ has a cusp singularity at $p\in {\rm supp\, D}$, then $p$ is a simple pole of $\omega$ with residue $-\sqrt{-1}r$, where $r>0$.

\item If\ $\mathrm{d} s^2$ has a cone singularity at $p\in {\rm supp\, D}$ with angle $2\pi m\in 2\pi\,\mathbb{Z}_{>1}$, then $p$ is a zero of $\omega$ with order $m-1$.

\end{enumerate}

\end{lemma}

\begin{proof}
Take a point $p$ on $\Sigma\backslash {\rm supp\, D}$ and choose a function element $\mathfrak{f}$ of $f$ near $p$. Since $f$ has monodromy in $L_{0}$, $\omega={\rm d}\mathfrak{f}$ does not depend on the choice of the function element $\mathfrak{f}$. Hence $\omega={\rm d}f$ is a holomorphic one-form on $\Sigma\backslash {\rm supp\, D}$. We postpone to (1-2) the proof that it extends to a meromorphic one-form on $\Sigma$.

\begin{enumerate}

\item By Lemma \ref{lem:L},  $\mathrm{d} s^2$\ has cusp singularities or cone ones with angles lying in $2\pi\,\mathbb{Z}_{>1}$. Let $p\in {\rm supp\, D}$ be a cusp singularity of $\mathrm{d} s^2$. By the proof of Lemma \ref{lem:L}, we can choose a function element $\mathfrak{f}$ near $p$ and a complex coordinate $z$ near $p$ such that $\mathfrak{f}=-\sqrt{-1}r\log z+s$, where $r>0$. So $\omega={\rm d}{\frak f}=\frac{-\sqrt{-1}r}{z}{\rm d}z$ has a simple pole at $p$ with residue $-\sqrt{-1}r$.

\item Suppose that\ $\mathrm{d} s^2$ has a cone singularity at $p$ with angle $2\pi m\in 2\pi\,\mathbb{Z}_{>1}$. We can choose a function element $\mathfrak{f}$ near $p$ and a complex coordinate $z$ near $p$ such that $\mathfrak{f}=\frac{az^{m}+b}{cz^{m}+d}$ with $ad-bc=1$, and $\omega={\rm d}{\frak f}=\frac{mz^{m-1}}{(cz^{m}+d)^{2}}{\rm d}z$. Since ${\frak f}$ takes values in ${\Bbb H}$, we cansee that  $d\neq0$, $p$ is a zero of $\omega$ with order $m-1$, and $\lim_{z\to p}\, {\frak f}(z)=\frac{b}{d}\in\mathbb{H}$.

\end{enumerate}
\end{proof}

\begin{theorem}
\label{thm:no L_{0}}
There exists no $L_0$-metric on a non-hyperbolic Riemann surface.

\end{theorem}

\begin{proof}
Let $\mathrm{d} s^2$ be an $L_0$-metric\ $\mathrm{d} s^2$ on a Riemann surface $\Sigma$, and
$f:\Sigma\backslash {\rm supp\, D}\to \mathbb{H}$ a developing map of $\mathrm{d} s^2$ with monodromy in $L_{0}$.  Then $-\Im\, f$ is a negative non-constant harmonic function on $\Sigma\backslash {\rm supp\, D}$ and then it is also subharmonic. Since an isolated point is polar, by Lemma \ref{thm:TR95}, $-\Im f$ extends to a negative non-constant subharmonic function on $\Sigma$. Hence, $\Sigma$ is a hyperbolic Riemann surface.
\end{proof}

\begin{prop}
\label{prop:LD}
 {\it An $L$-metric on ${\Bbb D}$ actually has monodromy in $L_0$.}
\end{prop}

\begin{proof}

Let $\mathrm{d} s^2$ be an $L$-metric representing an ${\Bbb R}$-divisor D on the unit disc ${\Bbb D}$. Let $f:{\Bbb D}\backslash {\rm supp\, D}\to \mathbb{H}$ a developing map of $\mathrm{d} s^2$ with monodromy in $L$. By Lemma \ref{lem:L},  $f''/f'$ is a meromorphic function on ${\Bbb D}$ whose residues are all integers.
Hence, taking $z_0\in {\Bbb D}\backslash {\rm supp\, D}$, we find that
$
f'(z)=\exp\left(\int_{z_0}^{z}\frac{f''}{f'}{\rm d}z\right)+f'(z_0)
$
is a meromorphic function on ${\Bbb D}$, and then the
monodromy of $f$ lies in $L_{0}$.

\end{proof}


\begin{example}\label{exam:LLX}({\rm{\cite[Example 2.1]{LLX20}}})

Let $\sum_{j=1}^{\infty}a_{j}$ be a convergent series of positive numbers and $\{z_{j}\}_{j=1}^\infty\subset{\Bbb D}$ a closed discrete subset. Then $h(z):=\sum_{j=1}^{\infty}\frac{a_{j}}{z-z_{j}}$ is a meromorphic function on the unit disc ${\Bbb D}$ and there exists a real number $\lambda_{0}$ and a one-parameter family  of $L_{0}$-metrics representing the same $\mathbb{Z}$-divisor ${\rm D}=(h)$ on ${\Bbb D}$. Hence, these metrics have cusp singularities at $z_{j}$'s and a cone singularity of angle $2\pi\big(1+ {\rm ord}_{w}(h)\big)$ at a zero $w$ of $h$.
\end{example}

\section{Proof of Theorem \ref{thm:ZC} and some discussions}

\begin{proof}[Proof of Theorem \ref{thm:ZC}]
Let ${\rm d}s^2$ be a singular hyperbolic metric on a non-hyperbolic Riemann surface $\Sigma$.
The first statement follows from Theorems \ref{thm:H_{2}'no},  \ref{thm:noH_{3}} and \ref{thm:no L_{0}}.
The second one follows from the first one, Corollary \ref{Cor:2}, Theorem \ref{thm:no L on C^{*}},  and the classification of positive dimensional proper Lie subgroups of ${\rm PSL}(2,\,{\Bbb R})$ given in the introduction.

\end{proof}

We would like to make a discussion about Conjecture \ref{conj:ZP}, which claims that {\it a singular hyperbolic metric on a non-hyperbolic Riemann surface
has Zariski dense monodromy in ${\rm PSL}(2,\,{\Bbb R})$.} By Theorems \ref{thm:H_{2}'no},  \ref{thm:noH_{3}} and \ref{thm:no L_{0}}, the conjecture is reduced to proving
that {\it the monodromy of a singular  hyperbolic metric on a non-hyperbolic Riemann surface does not lie  in $L$}. To show its subtlety, let us consider a special parabolic Riemann surface -- the thrice punctured sphere
${\Bbb C}\backslash \{0,\,1\}$, whose maximal Abelian cover $\big({\Bbb C}\backslash \{0,\,1\}\big)^{\rm Ab}$ is a hyperbolic Riemann surface by
a theorem of McKean-Sullivan \cite{MS84}. Hence, the argument for the non-existence of $L$-metrics in the proof of Theorem \ref{thm:no L_{0}} becomes invalid on ${\Bbb C}\backslash \{0,\,1\}$.

Both Proposition \ref{prop:U(1)} and Example \ref{exam:LLX} show that the uniqueness of hyperbolic metric representing a given ${\Bbb R}$-divisor fails on ${\Bbb D}$. However, the hyperbolic metrics
there are not complete in general even after adding their cone singularities. We come up with the following.\\

\noindent {\bf Question 6.1.} Let ${\rm D}=\sum_{n=1}^\infty\, (\theta_j-1)p_j$ be an ${\Bbb R}$-divisor on ${\Bbb D}$ such that $0\leq \theta_j\not=1$ and $\{p_n\}$ is a discrete closed subset of ${\Bbb D}$. When does there exist a singular hyperbolic metric ${\rm d}s^2$ representing D such that it extends to a complete metric on ${\Bbb D}\backslash \{p_n:\, \theta_n=0\}$? If yes, would it be unique? Hulin-Troyanov \cite[Theorem 8.2]{HT92} answered these two questions affirmatively if ${\rm supp\, D}$ is a finite subset of ${\Bbb D}$.\\

\noindent\textbf{Acknowledgement:} The authors would like to thank Andriy Haydys, Rafe Mazzeo and Song Sun for many insightful discussions, and their hospitality during B.X.'s visit to
Freiburg University in Winter 2018 and Summer 2019, Stanford University and UC Berkeley in Spring 2019.
The authors would also like to thank Martin de Borbon, Alexandre Eremenko, Qiongling Li, Dmitri Panov,  Yi Qi, Mao Sheng and Ming Xu for helpful suggestions and comments. Part of the work was done while B.X. was visiting Institute of Mathematical Sciences at ShanghaiTech University in Spring 2019.


\bibliographystyle{amsplain}

\end{document}